\documentclass[10pt,a4paper]{amsart}
\usepackage{amsmath,amssymb,epsf,graphics}
\usepackage{amsfonts,amsxtra,amscd,verbatim,eucal}
\usepackage{enumerate}

\newtheorem{theorem}{Theorem}[section]
\newtheorem{lemma}[theorem]{Lemma}
\newtheorem{proposition}[theorem]{Proposition}
\newtheorem{corollary}[theorem]{Corollary}

\theoremstyle{definition}
\newtheorem{definition}[theorem]{Definition}

\theoremstyle{remark}
\newtheorem{rmk}[theorem]{Remark}

\newtheorem*{claim}{Claim}
\newtheorem{assn}[theorem]{Assumption}
\newtheorem*{ack}{Acknowledgments}

\newtheorem{theoremalpha}{Theorem}

\newtheorem{definitionalpha}[theoremalpha]{Definition}

\def\R{{\mathbb R}}
\def\Q{{\mathbb Q}}

\def\OO{{\mathcal O}}

\DeclareMathOperator{\bdiv}{\b{div}}

\newcommand{\HH}[3]{H^{{#1}} \big( {#2} , {#3}
\big) }

\newcommand{\andd}{ \ \ \text{ and } \ }

\newcommand{\deq}{\ensuremath{ \stackrel{\textrm{def}}{=}}}
\newcommand{\equ}{\ensuremath{ \,=\, }}
\newcommand{\mif}{\ensuremath{\textrm{\ if\ }}}
\renewcommand{\b}[1]{\ensuremath{\mathbf{#1}}}
\newcommand{\st}[1]{\ensuremath{ \left\{ #1 \right\} }}
\newcommand{\spn}[1]{\ensuremath{ \langle #1 \rangle }}

\DeclareMathOperator{\Fix}{Fix}
\DeclareMathOperator{\Supp}{Supp}
\DeclareMathOperator{\WDiv}{WDiv}

\DeclareMathOperator{\Zero}{Zero}
\DeclareMathOperator{\Weil}{Weil}

\begin{document}

\title{Zariski decomposition of b-divisors}

\author{Alex K\"uronya}
\address{Universit\"at Duisburg-Essen, Campus Essen, FB 6 Mathematik, D-45117 Essen, Germany}

\address{Budapest University of Technology and Economics, Budapest P.O. Box 91, H-1521 Hungary}
\email{{\tt alex.kueronya@uni-due.de}}

\author{Catriona Maclean}
\address{Universit\'e Joseph Fourier, UFR de Math\'ematiques, 100 rue des Maths, BP 74, 38402 St Martin d'H\'eres, France}
\email{\tt Catriona.Maclean@ujf-grenoble.fr}

\thanks{During this project the first author was partially supported by  the  DFG-Leibniz program, the SFB/TR 45 ``Periods, moduli spaces and arithmetic of algebraic varieties'', and the OTKA Grant 61116 by the Hungarian Academy of Sciences. The second author obtained partial support from the BUDALGGEO Marie Curie Host Fellowship for the Transfer of Knowledge and the ANR project 3AGC.}

\maketitle

\begin{abstract}
Based on a recent work of Thomas Bauer's \cite{Bauer} reproving the existence
of Zariski decompositions for surfaces, we construct a  b-divisorial
analogue of Zariski decomposition  in all dimensions.
\end{abstract}

\section{Introduction}

The purpose of this paper is to present a generalization of Zariski 
decomposition on surfaces to the context of b-divisors. In particular, 
we provide such a  decomposition for an arbitrary effective
 $\Q$-b-divisor on a normal $\Q$-factorial projective variety in the sense of 
b-divisors. 

Originating in the seminal work of Zariski \cite{Zariski} on the structure 
of linear systems of surfaces, the Zariski decomposition  $D=P_D+N_D$ of an effective $\Q$-divisor $D$ on a 
smooth projective surface  $X$ over an  algebraically closed field  consists of a nef divisor $P_D$ and a negative 
cycle $N_D$ satisfying an orthogonality condition with respect to the 
intersection form on $X$. More specifically, given any effective 
$\Q$-divisor $D$, Zariski proves that there is a unique decomposition of $D$ 
\[
D=P_D+N_D
\] 
such that  $P_D$ is nef and $N_D$ is effective;  $P\cdot C=0$, for any curve $C$ appearing in $\Supp (N)$; and if  
$\Supp (N)= C_1\cup\dots \cup C_n$ then  the intersection matrix $I(C_1,\ldots C_n)$ is negative definite.
Zariski decomposition has  the following useful properties.
\begin{enumerate}
\item For any integer $k$, one $\HH{0}{X}{\lfloor
kP\rfloor} = \HH{0}{X}{\lfloor kD\rfloor}$ (ie. 
$P_D$ "carries all sections of $D$").
\item If the effective nef divisor $P'$  satisfies $P'\leq D$ then
$P'\leq P_D$\ ,
\end{enumerate} 
Thus providing a strong tool to understand linear series on  surfaces. It has been playing a distinguished role in the theory ever since, among others it  is very useful for studying section rings 
\[
R(X,D) \deq \oplus_{m\geq 0}\HH{0}{X}{\lfloor mD\rfloor} \ .
\]
Since $R(D)=R(P_D)$, Zariski decomposition  allows us to 
reduce questions concerning $R(D)$ --- most notably whether it is finitely generated ---
to the case where $D$ is nef. As an illustration Zariski's paper  contains an appendix by Mumford which uses several of Zariski's  results in \cite{Zariski} to prove that the canonical ring of a surface of general type is finitely generated.

There is no immediate way to extend this definition to  higher-dimensional varieties. 
Apart from an earlier attempt by Benveniste (\cite{Ben1}, \cite{Ben2}) 
all proposed higher-dimensional generalizations have been based on those 
properties of the Zariski decomposition which make it useful for studying 
section rings, namely a) $P_D$ is nef and b) $\HH{0}{X}{\lfloor kP_D\rfloor} =
\HH{0}{X}{\lfloor kD\rfloor}$ for all $k$.

Given an effective $\Q$-divisor $D$ on a  variety $X$ it is easy to see that no $P_D\leq D$ can satisfy both properties if $D$  is non-nef but some multiple $kD$ of $D$ has no divisorial fixed locus. 
To get around this problem, we allow blow-ups. In \cite{Kawamata}, (see also \cite{moriwaki}) Kawamata defines a Zariski-type decomposition as follows.

\begin{definitionalpha}
Let $D$ be a big divisor on a normal variety $X$. A rational (resp. real) 
Zariski decomposition of $D$ in the sense of Cutkosky--Kawamata--Moriwaki  is a 
proper birational map $\mu:\widetilde{X}\to X$, and an  effective $\Q$ (resp. $\R$)
divisor $N_D\leq \mu^*D$ such that 
\begin{enumerate}
\item $P_D=\mu^*D-N_D$ is nef,   
\item  $\HH{0}{X}{\lfloor \mu^*(kD)\rfloor }= \HH{0}{X}{\lfloor kP_D \rfloor}$
for all $k\geq 1$.
\end{enumerate}
\end{definitionalpha}
The requirement that  the divisor $D$ be big forces the  
Zariski decomposition, if it exists, to be  unique up to birational modification.
Indeed, we then have 
\[
N_D=\lim_{n\rightarrow \infty} \frac{{\rm Fix}(mD)}{m}\ ,
\] 
by a result of Wilson's  (see \cite[Theorem 2.3.9]{PAG} or \cite{Wilson}).
The inclusion of real Zariski decompositions, hitherto believed to be 
uninteresting since real divisors very rarely have finitely generated section 
rings, is motivated by a counterexample of Cutkosky's \cite{Cutkosky} showing that 
certain divisors only have real Zariski decompositions.

As the main result of \cite{Kawamata}, Kawamata proves that if $(X,\Delta)$ is a normal klt pair such that $K_X+\Delta$ is big and possesses a {\it real } Zariski decomposition then its log canonical  ring is finitely generated. 

A subsequent counterexample of Nakayama's  \cite{Nakayama} showed that in general even real Zariski decompositions  do not exist on higher-dimensional varieties.

The conditions a) and b) do not define a decomposition 
$D=P_D+N_D$ uniquely if $D$ is not big, even on surfaces. 
For example, let $E$ be an elliptic 
curve and set $X={\rm Proj}(\mathcal{O}_E\oplus L)$, where $L$ is any
degree-zero  non-torsion line bundle on $E$. If we take $D={\rm Proj} (L) \subset
X$ then $D$ is nef but $H^0(kD)=\mathbb{C}$ for any $k$, so for any rational
$0\leq \lambda\leq 1$ the decomposition $P_D=\lambda D$ and 
$N_D=(1-\lambda) D$ satisfies conditions a) and b). Fujita gets around this 
problem by using the maximality  of $P_D$ amongst nef sub-divisors of $D$. In \cite{fujita} he gives 
the following definition.
\begin{definitionalpha}
Let $D$ be an effective $\Q$-divisor on a normal variety $X$. 
A rational (resp. real) 
Fujita Zariski decomposition of $D$ is a 
proper birational map $\mu:\widetilde{X}\to X$, and an 
effective $\Q$ (resp. $\R$)
divisor $N_D\leq \mu^*D$ such that 
\begin{enumerate}
\item $P_D=\mu^*D-N_D$ is nef,   
\item For any proper birational map $\nu: X''\rightarrow X'$ and nef effective 
divisor $P''\leq \nu^*\mu^*(D)$ we have that $P''\leq \nu^*(P_D)$. 
\end{enumerate}
\end{definitionalpha}

A Zariski decomposition in the sense of Fujita is automatically a  Zariski 
decomposition in the sense of Cutkosky--Kawamata--Moriwaki.

The advent of multiplier ideals brought a certain analytic version of this 
concept. Precisely, Tsuji defines in \cite{Tsuji1}
\begin{definitionalpha} 
Let $L$ be a line bundle on $X$, a variety. An analytic Zariski decomposition
of $L$ is a singular metric $h$ on $L$, semipositive in the sense of currents, 
such that
for all $k$, $\HH{0}{X}{L^k}=\HH{0}{X}{L^k\otimes \mathcal{I}(h^{\otimes k})}$.
\end{definitionalpha}

The motivation for this definition is as follows. Suppose that there were
a Zariski decomposition $\mu^*(D)=P_D+N_D$ for $D$ on some birational 
modification $\mu: X'\rightarrow X$. If the line bundle $\mathcal{O}(P_D)$ 
were not only nef, 
but actually semi-positive (a slightly stronger condition, which implies 
nef and is implied by ample), then we can put a semi-positive smooth metric on
$\mathcal{O}(P_D)$. 
This descends to a singular semi-positive metric on $L=\mathcal{O}(D)$, 
and by definition of the analytic multiplier ideal, a section 
$\sigma\in \HH{0}{X}{mD}$ is contained in $\HH{0}{X}{mD\otimes\mathcal{I}(h^{\otimes m})}$
if and only if $\mu^*(\sigma)$ is contained in  $\HH{0}{X}{mP_D}$.
Note that the analytic Zariski decomposition is not unique.
More importantly, it is considerably weaker than its algebraic counterpart. The fact that $K_X$ has an analytic decomposition does not imply that the canonical ring is finitely generated. 

In \cite{Demailly}, Demailly, Peternell and Schneider 
prove the following theorem: given a pseudo effective line bundle $L$ on a complex variety 
$X$, $L$ admits an analytic Zariski decomposition. Up to equivalence of singularities, the set of
analytic Zariski  decompositions admits a unique minimally singular member.  Our work here  can be seen as an algebraic version 
of this result. In \cite{bfj}, Boucksom, Favre and Johnsson consider a construction called the positive intersection product of
a set of $b$-divisors: in the case  where the set contains only one element, this is a 
Zariski-type decomposition. Their definition is the same as ours for big divisors.

Shokurov's paper \cite{shokurov} and the survey article by Prokurhov 
accompanying it \cite{prokhorov} contain many interesting Zariski-type
decompositions,
some of which work for $b$-divisors. In particular, the decomposition
$\b{D}= \b{D}^m+\b{D}^e$ defined in example $4.30$ of Shokurov's paper
gives us a Zariski-type 
decomposition for $b$-divisors. Although the definition is different,
we show below that the our definition gives the same result as
Shokurov's in the case where the divisor is big. However, they differ
for non-big $b$-divisors.  

The recent paper \cite{bfj} by Boucksom, Favre and Jonsson also
includes a Zariski-type decomposition of $b$-divsors. As discussed
in section 3.4 of their paper, the case $n=1$ of their
positive intersection products give a Zariski-type decomposition of
$b$-divisors. Their definition is the same as ours in the case of
big $b$-divisors.

In his original proof, Zariski concentrated on constructing the negative 
part $N_D$ using cunning linear algebra, which made for a reasonably 
complicated proof. In a recent work Bauer \cite{Bauer} gave a 
conceptual and  very elegant construction of Zariski decompositions on 
surfaces  using the 
characterization of the nef part $P_D$ as the maximal nef subdivisor of $D$. 

It is this latter approach that we use to extend the notion of Zariski 
decomposition to b-divisors.  We retain most of the characteristics of the 
higher-dimensional case with one notable exception: the positive part of a 
b-divisor is 
only a limit of b-nef b-divisors in a suitable sense. Our main result, proven as Theorem \ref{thm:main}, 
is as follows.

\begin{theoremalpha}\label{thma:main}
Let $X$ be a $\Q$-factorial normal projective variety over an 
algebraically closed field of characteristic 0, $\b{D}$ an effective $\Q$-b-divisor 
on $X$. 
Then there exists a unique decomposition
\[
\b{D} = \b{P_D} + \b{N_D}\ ,
\]
where $\b{P_D}$,$\b{N_D}$ are  effective b-$\mathbb{R}$ divisors on $X$, 
such that
\begin{enumerate}
\item $\HH{0}{X}{\lfloor k\b{P_D}\rfloor } = \HH{0}{X}{\lfloor k\b{D}\rfloor }$
\item $\b{P_D}$ is a limit of b-nef b-divisors on every 
proper birational model $Y\to X$.
\item For any nef b-divisor $\b{P'}\leq
\b{D}$ we have that $\b{P'}\leq \b{P_D}$.
\end{enumerate}
\end{theoremalpha}

A  few words about the organization of this paper. First we fix notation, and review our strategy in Section 2. In Section 3 we construct the b-divisorial Zariski decomposition, and prove its properties modulo results proved later in the article. 
Section 4 is devoted to the construction  and properties of separating blow-ups, the  main technical tool of the paper.

\begin{ack}
Helpful discussion with Eckart Viehweg were appreciated, as were
comments by Sebastien Boucksom.  
\end{ack}

\section{Strategy and overview.}

All varieties are normal  projective varieties, unless otherwise mentioned. 
An integral b-divisor $\b{D}$  on $X$  is an element of the group
\[
\b{Div}(X) \deq  \lim_{\leftarrow} \WDiv (Y)\ , 
\]
with  the limit  taken over all proper birational models $f : Y \rightarrow X$ along  with the induced
homomorphisms $f_* : \WDiv(Y )\to \WDiv(X)$. To put it differently,  $\b{D}$ is a collection of
divisors $D_Y \in \WDiv(Y )$ compatible with push-forwards. For every $Y\to X$, $D_Y$ is called the trace of 
$\b{D}$ on $Y$, and is denoted $\b{D}_Y$. One obtains rational and real b-divisors by tensoring with $\Q$ and $\R$,  respectively. 

As usual, the b-divisor of a nonzero rational function $\phi\in k(X)$ is defined as 
\[
 \b{\bdiv}_X (\phi) \deq \sum_{\nu}{\nu_E(\phi)E}
\]
where $E$ runs through all geometric valuations with center  on $X$. Two
b-divisors are considered linearly equivalent if they differ by the b-divisor of a
nonzero rational function.

One defines the associated b-divisorial sheaf $\OO_X(\b{D})$  by
\[
\Gamma (U,\OO_X(\b{D}))  \deq  \st{\phi\in k(X)\, |\, (\b{\bdiv}_X \phi + \b{D})|U \geq  0} \ .
\]
Note that the sheaf $\OO_X(\b{D})$ is \emph{not} coherent, however  its space of global sections is  finite-dimensional due to the inclusion  $\OO_X(\b{D})\hookrightarrow \OO_X(\b{D}_X)$.
The Cartier closure of an $\R$-Cartier divisor $D$  on $X$  is the b-divisor $\overline{D}$  with trace $(\overline{D})_Y \deq  f^*D$ on every model $f:Y\to X$.
For more on the language of b-divisors the reader might wish to consult  the appropriate chapter of \cite{Corti}.

In constructing Zariski decompositions for b-divisors, we will follow the approach of Bauer \cite{Bauer}. To this end, 
we start by reviewing his proof for the surface case. Given an effective divisor $D$ on a 
surface $X$, Bauer sets:
\[
P_D \deq \max \{ P'\leq D, P {\rm\ nef}\}\ .
\]
By this maximum, we mean that the coefficient of a prime divisor $E$ in 
$P_D$ is the maximum of $c_{E}(P')$ of the coefficients of $E$ in nef 
subdivisors of $D$.

Assume for the moment that $P_D$ is itself nef. 
Set $N_D= D-P_D$, which is effective by construction. If $C\in \Supp (N)$ 
and $P_D\cdot C>0$ then for any small positive $\epsilon$, $P_D+\epsilon C$ is 
still a nef subdivisor of $D$, 
contradicting the maximality of $P_D$. If $I(C_1\ldots C_n)$ is not
negative definite then we can find an effective divisor $C'$ 
supported on ${\rm Supp}(N)$ such that $C'\cdot C_i\geq 0$ for all $i$. 
For small positive $\epsilon$, $P_D+\epsilon C'$ is then a nef subdivisor of 
$D$, contradicting the maximality of $P_D$.

The important point is therefore the nefness of $P_D$, which follows from the 
following lemma.
\begin{lemma}\label{lem:1}
Let X be a surface and let $D_1$ $D_2$ be two nef effective divisors on $X$.
Then ${\rm max}(D_1, D_2)$ is nef. 
 \end{lemma}

\begin{proof}
Let $C$ be an irreducible curve on $X$. We write $D_1= a_1C+ D'_1$ and $D_2=
a_2C+D'_2$. We may assume that $a_1\geq a_2$ so that ${\rm max}(D_1, D_2)=
D_1+M$, where the coefficient of $C$ in $M_1$ is $0$. Hence 
\[ 
{\rm max}(D_1, D_2)\cdot C=D_1\cdot C+M\cdot C\geq D_1\cdot C\geq 0\ .
\]
This completes the proof of the lemma.
\end{proof}

Obviously, this fails in higher dimensions, depending as it does on the fact 
that there is at most one prime divisor on a surface intersecting a given 
irreducible curve negatively. Our aim will be to show that on a suitable 
birational modification however, the statement of Lemma~\ref{lem:1} remains true for $\mathbb{Q}$-divisors. 
This will enable us to construct Zariski decompositions for b-divisors in arbitrary dimensions.

\begin{theorem}\label{thm:th1}
Let $X$ be a normal $\Q$-factorial variety, $D$ be an effective divisor on $X$,  $D_1$ and $D_2$  nef effective $\Q$-subdivisors of $D$. There exists a birational morphism $F: Y\rightarrow X$ such that 
\[
 \max (F^*(D_1), F^*(D_2))
\]
is nef. Moreover, we have that for any higher model $G: Z\rightarrow Y$, 
\[
G^*(\max (F^*(D_1), F^*(D_2))) \equ  \max(G^*F^*(D_1), G^*F^*(D_2))\ .
\]
\end{theorem} 
\noindent
We write  ${\rm Supp}(D)=\cup_i Q_i$, where the $Q_i$ are prime divisors, and  say that 
\[
Q_i \textrm{ is\ \ \  } \begin{cases} \textrm{ of type $1$, } & \mif {\rm coeff}_{D_1}Q_i\, >\, {\rm coeff}_{D_2} Q_i \\ 
	\textrm{ of type $2$, } & \mif {\rm coeff}_{D_1}Q_i\, <\,  {\rm coeff}_{D_2} Q_i \\
	\textrm{ of type $0$, } & \mif {\rm coeff}_{D_1}Q_i\equ  {\rm coeff}_{D_2} Q_i\ .
  \end{cases}
\]
Our proof of Theorem~\ref{thm:th1}  will be based on the following criterion. 

\begin{proposition}\label{prop:pr1}
The divisor ${\rm max}(D_1, D_2)$ is nef
if $Q_i\cap Q_j=\emptyset$ whenever $Q_i$ is of type 1 and $Q_j$ is of type 2.
Moreover, if this condition holds then for any higher model 
$G: Z\rightarrow X$ then
\[
G^*(\max (D_1, D_2)) \equ  \max(G^*(D_1), G^*(D_2))
\]
\end{proposition}

\begin{proof}
We prove first that $\max(D_1, D_2)$ is nef.
Let $C$ be a curve. We note that at least one of the following holds:
\begin{itemize}
\item  There is no $Q_i$ of type 1 containing
$C$,
\item There is no $Q_i$ of type 2 containing $C$. 
\end{itemize}
Without loss of generality 
there is no $Q_i$ of type 2 containing $C$. We can write
 \[{\rm max}(D_1,D_2)= D_1+\sum_{i, Q_i \mbox{\small of type 2}} c_iQ_i,\]
where the $c_i$s are positive constants. But then 
\[C\cdot {\rm max}(D_1, D_2) =C\cdot D_1 + 
C\cdot  (\sum_{i, Q_i \mbox {\small of type 2}}
c_iQ_i)\]
Since $C$ is contained in no $Q_i$ of type 2, the final
 term is positive and $C\cdot D_1\geq 0$ because $D_1$ is nef. Hence  
$C\cdot {\rm max} (D_1, D_2)\geq 0$ for any curve $C$. 

Now, let us prove that given a birational map $G: Z\rightarrow X$ we have that
\[
G^*(\max (D_1, D_2)) \equ \max(G^*(D_1), G^*(D_2)) \ .
\]
We consider a divisor $E\subset Z$. For each $Q_i$, let $d_i\geq 0$ 
be the coefficient of $E$ in $G^*(Q_i)$. We may assume that $\forall i$ 
such that $d_i\neq 0$ $Q_i$ is of type 1 or 0. Writing 
\[ 
D_1= \sum_i a_i Q_i\ ,\ D_2=\sum_i b_i Q_i
\]
the coefficient of $E$ in $G^*(\max(D_1, D_2))$ is $= 
\sum_i d_i\max(a_i,b_i)=\sum_i d_ia_i$ since whenever $d_i\neq 0$ $a_i\geq b_i$.
But the coefficient of $E$ in $\max(G^*(D_1), G^*(D_2))=
\max(\sum_i a_id_i,\sum_i b_id_i)=\sum_i a_id_i$ and hence the condition holds. 
\end{proof}
  
To establish Theorem~\ref{thm:main}, it will therefore be enough to prove the
following result.

\begin{theorem}\label{thm:th2}
Let $X,D_1,D_2$ be as above. There is a projective birational morphism 
\[
F:Y\longrightarrow X
\] 
from a normal $\Q$-factorial variety $Y$ with the following property. 
If  \[F^{-1}(D)=\cup_i Q_i,\] then
for any pair $i,j$ such that  $Q_i$ is of type $1$ and $Q_j$
is of type 2 we have that
$Q_i\cap Q_j=\emptyset.$\\ \\ (Here ``of type 1'', for example, is to be 
understood with respect to the pair of divisors $F^*(D_1)$ and $F^*(D_2)$.)  
\end{theorem} 

\noindent
We say that $(Q_i,Q_j)$ is a {\it bad pair}
if $Q_i$ is of type 1, $Q_j$ is of type
2 and $Q_i\cap Q_j\neq \emptyset$. 
We note that if $Q_i$ is of type 1 (resp.2, resp. 0) in $X$ 
then the proper transform
$\overline{Q}_i$ is also of type 1 (resp. 2 resp. 0) in $Y$. 
Our aim will therefore be to create a blow-up
$\hat{X}'$ of
$X$ along $Q_i\cap Q_j$ for any bad pair $(i,j)$
such that 
\begin{itemize}
\item $\overline{Q_i}$ and $\overline{Q_j}$ are separated in $\hat{X}'$ and
\item the unique exceptional divisor $E\subset \hat{X}'$ is of type 0.
\end{itemize}

We relegate the proof of Theorem~\ref{thm:th2} to Section~\ref{sec:sep_blowup}.

\section{Construction of Zariski decomposition for b-divisors}

We proceed with the actual construction of Zariski decompositions, and prove our main result. 

\begin{theorem}\label{thm:main}
Let $X$ be a $\Q$-factorial normal projective variety over an 
algebraically closed field, $\b{D}$ an effective $\Q$-b-divisor 
on $X$. 
Then there exists a unique decomposition
\[
\b{D} = \b{P_D} + \b{N_D}\ ,
\]
where $\b{P_D}$,$\b{N_D}$ are  effective $\mathbb{Q}$-b-divisors on $X$, 
such that
\begin{enumerate}
\item $\HH{0}{X}{\lfloor k\b{P_D}\rfloor } = \HH{0}{X}{\lfloor k\b{D}\rfloor }$
\item $\b{P_D}$ is a limit of b-nef b-divisors on every 
proper birational model $Y\to X$.
\item For any nef $b$-divisor $\b{P'}\leq
\b{D}$ we have that $\b{P'}\leq (\b{P_D})$.
\end{enumerate}
\end{theorem}

Granting  Theorem~\ref{thm:th2}, we show how to prove  Theorem~\ref{thm:main} and construct a 
Zariski decomposition in the sense of b-divisors. We start by recalling the 
definition of a nef b-divisor.

\begin{definition}
Let $\b{P}$ be a b-divisor on $X$. We say that $\b{P}$ is nef if 
there is a birational model $X'\rightarrow X$ such that 
\begin{itemize}
\item $P \deq \b{P}_{X'}$ is nef, 
\item $\b{P}=\overline{P}$, the Cartier closure of $P$.
\end{itemize}
\end{definition}

We are now going, given a $\Q$-b-divisor $\b{D}$ on $X$, to define the positive part of $\b{D}$. 

\begin{definition}
We set 
\[ 
\b{P_D} \deq  {\rm max}\{\b{P}| \b{P} \mbox{ a nef $\Q$-b-divisor }, \b{P}\leq \b{D}\} \ .
\] 
\end{definition}

After finishing this paper, we learnt that a very similar construction
has been used by Boucksom, Favre and Jonsson in their paper
\cite{bfj}. More precisely, in the case where $\b{D}$ is a Cartier big
divisor, our definition is the same as Boucksom, Favre and Jonsson's
definition of a positive intersection product, in the special case
where the set of multiplied divisors contains only one element. (See
Definition 2.5 and section 3.4 in \cite{bfj} for more information.)

Then $\b{P_D}$ is a well-defined b-divisor on $X$, and 
$0\leq \b{P_D}\leq \b{D}$. In order to prove Theorem~\ref{thm:main}, we need two preliminary lemmas.

\begin{lemma}\label{lem:maxnef}
Let $\b{P}_1$ and $\b{P}_2$ be nef $\Q$-b-divisors. Then $\max(
\b{P}_1,\b{P}_2)$ is again a nef $\Q$-b-divisor.
\end{lemma}

\begin{proof}
After suitable blow-up, we may assume that 
$\b{P}_i= \overline{P_i}$
is the Cartier closure of a nef divisor $P_i$ on $X$. Theorem~\ref{thm:th1} says that   
we may further assume that 
\[
P \deq \max(P_1, P_2)
\] 
is nef and that on any higher model 
$G:Z\rightarrow X$ \[
\max(G^* P_1, G^* P_2)=G^*(P).
\] 
Alternatively, $\max(\b{P_1}, \b{P}_2)=\overline{P}$ and hence 
$\max(\b{P}_1, \b{P}_2)$ is a  nef divisor.
\end{proof}

Throughout the following, set
\[ 
\b{M}_k(\b{D}) \deq \b{D}-\frac{1}{k}\Fix(k\b{D})
\]
\begin{lemma}\label{lem:mknef}
Let $\b{D}$ be a $\mathbb{Q}$-b-divisor on $X$. Then $\b{M}_k(\b{D})$ is a 
nef $\mathbb{Q}$-b-divisor.
\end{lemma}

\begin{proof}
Let $D$ be the trace of $\b{D}$ on $X$. Set 
$V=H^0(X,\lfloor k\b{D}\rfloor )\subset H^0(X, \lfloor kD\rfloor )$. By Hironaka's resolution of singularities there is a model 
$F: Y\rightarrow X$ such that the mobile part of the linear system $V$ on $Y$ 
is base-point-free, ie. we can write
\[
k\b{D}_Y= M_Y +F_Y 
\]
in such a way that $V\subset H^0(Y,M_Y)$ and $V$ is base point free as a linear system in   $H^0(Y,M_Y)$. We note that $ \b{M}_k(\b{D})_Y= \frac{1}{k}M_Y$ and that $M_Y$ is nef.
Since $V\subset H^0(Y,M_Y)$ is base-point free we have that
\[
{\rm Fix}(k\b{D})_Z= k\b{D}_Z-G^* (M_Y)
\] 
on any higher model $G:Z\rightarrow Y$ and hence 
\[
\b{M}_k(\b{D})=\overline{M_Y}
\] 
It follows that $\b{M}_k(\b{D})$ is a nef b-divisor. 
\end{proof}

\begin{proof}[Proof of Theorem~\ref{thm:main}.]
First we prove that we have $H^0(X, \lfloor k\b{P_D}\rfloor )= H^0(X,\lfloor k\b{D}\rfloor )$ for any $k$.
By Lemma~\ref{lem:mknef}, $\b{M}_k(\b{D})$  is a nef $\mathbb{Q}$-b-divisor. It follows 
by definition that $\b{P_D}\geq \b{M_k}(\b{D})$ and hence that
\[ 
H^0(X, \lfloor k\b{P_D}\rfloor )=H^0(X,\lfloor k\b{D}\rfloor )\ .
\]
Condition 3 is satisfied by definition of $\b{P_D}$. It remains to 
prove condition 2. Here is what we will prove.

\begin{claim}
For any birational model $X'\rightarrow X$, there is a sequence $\b{N}_n$ of nef 
$\mathbb{Q}$-$b$-divisors 
such that $\lim_{n\rightarrow \infty}(\b{N}_n)_{X'}=(\b{P_D})_{X'}$
\end{claim}

To this end, set $\b{P_D}_{X'}=\sum_i c_i Q_i$, 
the sum being taken over some finite set of 
irreducible divisors. Let $\epsilon$ be a positive real number. It will be 
enough to show that there is some nef b-divisor $\b{N}_\epsilon$ 
such that 
\[
 || (\b{P_D})_{X'}-(\b{N})_{X'}||\leq \epsilon
\]
in the supremum norm.
Set $c_i-\epsilon= d_i$. By definition of $\b{P_D}$ there 
exists  for every $i$ a nef divisor $\b{N}_i$ on a model $X_i$ such that
$\b{N}_i\leq \b{D}_{X_i}$, 
and ${\rm coeff}_{N_i}(Q_i)\geq d_i$. 

By Lemma~\ref{lem:maxnef}  
\[
\b{N}_\epsilon=\max (\overline{N}_i)
\] 
is nef and $\leq \b{D}$. It then follows that
$\b{N}_\epsilon\leq \b{P_D}$. 
Since we also have ${\rm coeff}_{N}(Q_i)\geq c_i -\epsilon$ for all 
$i$, it follows that 
\[
 || (\b{P_D})_{X'}-(\b{N}_\epsilon)_{X'} ||\leq \epsilon \ ,
\]
and the Theorem follows.
\end{proof}

Since in the case of smooth surfaces there is no need for birational modifications, we get back the Cartier closure of the original Zariski decomposition. Going to higher dimensions, by uniqueness we obtain the following.

\begin{corollary}
 Let $D$ be a $\Q$-Cartier divisor on $X$ having a Zariski decomposition $D=P_D+N_D$ 
in the sense of Fujita, let $\overline{D}=\b{P}_{\overline{D}}+\b{N}_{\overline{D}}$ be the b-divisorial Zariski decomposition 
of the Cartier closure $\overline{D}$. Then 
\[
\b{P}_{\overline{D}} \equ \overline{P_D}\ , \andd \b{N}_{\overline{D}} \equ \overline{N_D}\ . 
\]
 \end{corollary}

In the special case where $\b{D}$ is big, we can do better. Recall that a b-divisor is called big if it is the Cartier closure of a big divisor on some model.
\begin{proposition}
If $\b{D}$ is big, then 
$\b{P_D}=\lim_{m\rightarrow \infty} \b{M}_m(\b{D})$.
\end{proposition}

\begin{proof}
Let $\b{N}$ be nef and $\leq \b{D}$. Choose a $k$ such that 
$\b{M}_k(\b{D})$ 
is big as well as nef. By Lemma \ref{lem:maxnef},
$\b{M'}_k(\b{D})={\rm max} (\b{N}, \b{M}_k(\b{D}))$ is big and nef.
Blowing-up, we may assume that $\b{M'}_k(\b{D})$ is the Cartier closure of its trace on $X$,
$\overline{M'_k(D)}$.  
By Wilson's result \cite{Wilson} 
\[
\lim_{m\rightarrow \infty} (\b{M}'_k(\b{D}) 
-\frac{1}{m}\Fix (m\b{M}'_k(\b{D})))=\b{M}'_k(\b{D}).
\]
But now
\[\b{N}\leq\b{M}'_k(\b{D})= \lim_{m\rightarrow \infty}(\b{M}'_k(\b{D}) 
-\frac{1}{m}\Fix (m\b{M}'_k(\b{D})))\leq 
\lim_{m\rightarrow \infty}\b{M}_m(\b{D})\leq \b{P_D}
\] where the last inequality is valid because $\b{P_D}$ carries all the sections of $\b{D}$. Hence 
\[
\b{N}\leq \lim_{m\rightarrow \infty}\b{M}_m(\b{D})\leq \b{P_D}
\]
for any nef sub-divisor $\b{N}$ of $\b{D}$.
Since $\b{P_D}$ is simply the maximum of all such $\b{N}$'s, it follows that
\[
P_D=\lim_{m\rightarrow \infty}\b{M}_m(\b{D})\ .
\]
In this case, in particular, $\b{P_D}$ is a limit in the strong 
sense of nef $b$-divisors.
\end{proof}

Although the positive part of a b-divisorial Zariski decomposition is not nef, it shares many of the important properties of nef divisors, vanishing being one of the most important.

\begin{corollary}[Vanishing Theorem]
Let $\b{D}$ be a big b-divisor on a smooth variety $X$. then 
\[
 \HH{i}{X}{\OO_X(K_X)\otimes \OO_X(\b{P_D})} \equ 0
\]
for all $i\geq 1$.
\end{corollary}
\begin{proof}
In this case, we have $\OO_X(-\b{N_D})=\mathcal{I}(||D||)$, the multiplier ideal of $D$, so this is just another restatement of Nadel vanishing.
\end{proof}

\section{The blow-up separating $Q_i$ and $Q_j$.}\label{sec:sep_blowup}

We move on to proving Theorem~\ref{thm:th2}, the technical core of the paper. 
It will be useful to change conventions slightly: from now on, the set of
divisors 
\[
S(X,D_1,D_2)\equ \{Q_1\ldots Q_r\}
\] 
will consist of all divisors in the support of
$D$ which are of type 1 or type 2. In other words, we remove from this set all
the divisors of type 0. 

We make the following assumptions. 
\begin{assn}\label{assn:1-4}\ 
\begin{enumerate}
\item $X$ is a $\Q$-factorial normal variety.
\item For any $m$-tuple $(k_1, k_2,\ldots, k_m)$ the intersection
  $Q_{k_1}\cap\ldots \cap Q_{k_m}$ is of pure codimension $m$.
\item For any pair of distinct $(i,j)$ and for a sufficiently general point 
 $x\in Q_i\cap Q_j$, $x$ is a smooth point of $Q_i$, $Q_j$ and $X$ and 
$Q_i$ and $Q_j$ intersect transversally at $x$. 
\item For any pair of distinct $(i,j)$, $Q_i\cap Q_j$ is irreducible.
\end{enumerate}
\end{assn}

\begin{proposition}\label{prop:pr2}
Under assumptions \ref{assn:1-4}, for any bad pair $(Q_i,Q_j)$,
 there is a proper birational morphism 
$\mu:\tilde{X} \rightarrow X$ with a unique exceptional divisor $E$ such that:
\begin{enumerate}
\item $E$ is of type 0 (relative to $\mu^*(D_1)$ and $\mu^*(D_2)$), 
\item $\overline{Q}_i$ and $\overline{Q}_j$ do not meet in $\tilde{X}$,
\item the conditions of \ref{assn:1-4}  are valid for $\tilde{X}$ and 
$S(\tilde{X}, \mu^*(D_1), \mu^*(D_2))=\{ \overline{Q}_1\ldots \overline{Q}_r
\}$.
\end{enumerate}
\end{proposition}

Such a blow-up will be called a separating blow-up for $(Q_i, Q_j)$.
We start by showing that Theorem~\ref{thm:th2} follows immediately from Proposition~\ref{prop:pr2}.

\begin{proof}[Proof of Theorem~\ref{thm:th2}.]
After a possible initial blow-up, we may assume that conditions of \ref{assn:1-4}  are satisfied. 
Let $\mu: \tilde{X}\rightarrow X$  be a morphism whose existence is guaranteed by Proposition~\ref{prop:pr2}. Then the set of bad pairs 
for $(\tilde{X}, \mu^*(D_1), \mu^*(D_2))$ is a subset of
\[
\{ (\overline{Q}_{k_1}, \overline{Q}_{k_2})| (Q_{k_1}, Q_{k_2}) \mbox{
is a bad pair for} (X, D_1, D_2), (k_1, k_2)\neq (i,j) \}\ .
\]
In particular, the number of bad pairs strictly decreases under a separating
blow-up. Iterating this procedure, we produce a proper birational
map $F: Y\rightarrow X$ such that $(Y, F^*(D_1), F^*(D_2))$ has no bad 
pairs and $1-4$ holds for $Y$.  But then $F:Y\rightarrow X$ is exactly the map we seek in Theorem~\ref{thm:th2}
\end{proof}

To be able to proceed with the proof of Proposition~\ref{prop:pr2}, we start by defining the type of birational modification we need.  Let $a,b$ be positive coprime integers.
We now define an ``$(a,b)$ blow-up along the pair $(Q_i, Q_j)$''. 

Choose an integer $m$ such that $mQ_i$ and $mQ_j$ are
both Cartier. Denote
\[
 \pi: \hat{X} \deq \textrm{Bl}_{I_i^a+I_j^b}X\longrightarrow X
\]
where $I_i$ and $I_j$ denote the ideal sheaves $\OO_X(-mQ_i)$, and $\OO_X(-mQ_j)$, respectively.
As a consequence of \cite[Proposition 7.16.]{HS} $\hat{X}$ is a variety, in particular it is an integral scheme.

\begin{rmk}
The blow-up constructed above can be given explicitly in local terms as follows.
Choose open affines $U_k\subset X$ 
such that $(mQ_i)\cap U_k$ is defined by a single function $f_k$ and
$(mQ_j)\cap U_k$ is defined by a function $g_k$. 
We define  $\hat{U}_k$ of $U_k$ 
\[
\hat{U}_k \equ \{(x,[U:V])\in U_k\times \mathbb{P}^1| Ug_k^b=Vf_k^a\}\ ,
\]
where  we understand $\hat{U}_k$ to be the subscheme of 
$U_k\times \mathbb{P}^1$ defined by this equation. 
These open sets can  be glued together to give the global blow-up
scheme $\hat{X}$.
\end{rmk}

The variety $\hat{X}$ is not immediately useful, since it is not normal. 
Consider the normalisation $\hat{X}'$ of $\hat{X}$ in the function field
$K(\hat{X})$. We denote the normalisation map by $n: \hat{X}'\rightarrow \hat{X}$, and the composition
$\pi\circ n$ by $\pi'$. Note that $\pi$ is proper (since projective) and 
birational and that $n$ is proper and birational. Therefore, $\pi'$ is proper and birational. 

Throughout what follows the normalisation of an open set $A\subset\hat{X}$ will be denoted
by $A'\subset \hat{X}'$.
 
\begin{proposition}\label{prop:pr3}
The map $\pi':\hat{X}'\rightarrow X$ has a unique exceptional divisor.
\end{proposition}
\noindent
In the course of this proof, we will also find explicit equations 
for a certain open set $\hat{X}'$, which will be useful later on.

\begin{proof}[Proof of Proposition \ref{prop:pr3}.]
The exceptional locus of $\pi:\hat{X}\rightarrow X$ is a $\mathbb{P}^1$ 
bundle over the irreducible set $Q_i\cap Q_j$, so it only contains one 
exceptional divisor, call it $E_1$.

Any exceptional divisor in $\hat{X}'$ maps to $E_1$ under $n$. 
It will therefore be enough to show
that $n^{-1}(E_1)$ contains only one divisor. Moreover, since the
normalisation map is finite-to-one, 
it will be enough to find some open set, $\hat{U}\subset \hat{X}$,  meeting $E_1$, such that
$n^{-1}(E_1\cap U)$ contains a unique divisor in $n^{-1}(\hat{U})$. 

We choose an open affine set $W\subset X$ such that $W$ is smooth and $Q_i\cap W$ and
$Q_j\cap W$ are smooth and meet transversally. Such $W$ exists, since $\hat{X}'$ is normal. 
We assume further that there are regular functions $f$ and $g$ on $W$ such that $Q_i= \Zero (f)$
and $Q_j= \Zero (g)$.

One possible projective embedding of $\hat{W}=\pi^{-1}(W)$ is
\[
\hat{W} \equ\{(x, [U:V])\in W\times \mathbb{P}^1 | U g^{mb}=Vf^{ma}\}\ .
\] 
We consider the open affine set $\hat{U}\subseteq\hat{W}$  given by 
\[ 
\hat{U} \equ \{(x, [U:V])\in \hat{W}| U, V\neq 0\}\ ,
\]
which we can also write as
\[
\hat{U}\equ \{(x,u)\in W\times (\mathbb{A}^1\setminus \{0\})| ug^{mb}=f^{ma}\}\ .
\]
We note first of all that the rational function on $\hat{U}$ given by $s=
\frac{f^a}{g^b}$ satisfies $s^m=u$ and is hence a regular function on
$\hat{U}'$, the normalisation of $\hat{U}$.

Let us now consider the integral affine scheme
\[
\hat{U}_1 \equ \{(x,s)\in W\times (\mathbb{A}^1\setminus \{0\})| sg^b=f^a\}\ .
\]
There is a natural surjective map $\theta: \hat{U}_1\rightarrow \hat{U}$ given
by $\theta(x,s)=(x, s^m)$. This is an isomorphism over the open set
$\hat{U}\setminus E_1\cap \hat{U}$, so there is an inclusion
$A(\hat{U}_1)\subset K(\hat{U})$: since all elements of $A(\hat{U}_1)$ are
integral over $A(\hat{U})$ it follows that there are maps
\[
\hat{U}'\stackrel{\mu}{\rightarrow }\hat{U}_1 \stackrel{\theta}{\rightarrow} \hat{U}\ .
\]
such that $\theta\circ\mu=n$. 

In other words, $\hat{U}'$ is also the normalisation of $\hat{U}_1$, which 
however is still not normal: we need to add some extra regular functions.

Choose numbers $(c,d)$ such that $bd-ac=1$. Consider the element
\[
t \deq  \frac{f^d}{g^c}\in K(\hat{U}_1)\ .
\]
We note that 
\[
t^a \equ \frac{f^{ad}}{g^{ac}}=s^dg\ .
\] 
Similarly $t^b= s^cf$, and in particular $t\in A(\hat{U}')$. We now consider the scheme defined as follows.
\[
\hat{U}_2 \equ \{(x,s,t)\in \hat{U}\times (\mathbb{A}^1\setminus\{0\}) \times \mathbb{A}^1\,|\, t^a= s^dg, t^b= s^cf\}\ .
\]
In $\hat{U}_2$ we have that 
\[
s^{bd} g^b \equ  t^{ab} \equ s^{ac}f^a \equ s^{bd-1}f^a
\] 
and that
\[
s^{cd}g^{c}t \equ t^{ac+1} \equ t^{bd} \equ s^{cd}f^d\ .
\]
so it follows that $sg^b=f^a$ and $t g^c=f^d$ in $A(\hat{U}_2)$. 
In particular,
there is a natural map 
\[
\nu:\hat{U}_2\rightarrow \hat{U}_1
\]
given by $\nu(x,s,t)=(x,s)$. 
We note that
$\nu$ is surjective and set-theoretically one-to-one. 
Indeed, for any $(x,s)\in \hat{U}_1$, 
\[
\nu^{-1}(x,s) \equ \{(x,s,t)|t^a=s^dg, t^b=s^cf\}\ . 
\]
and it easy to see that for fixed $x,s$ such that $sg^b=F^a$ these equations have exactly one solution in $t$.
We note further that as sets 
\[ 
E_2 \deq (\theta\circ \nu)^{-1}(E_1) \equ (Q_i\cap Q_j\cap W)\times (\mathbb{A}^1\setminus \{0\})\times \{0\}\ ,
\]
and hence this set contains only one divisor. We aim now to show that 
$\hat{U}_2$ is in fact the normalisation of $\hat{U}$.
\begin{lemma}\label{lem:lem2}
$\hat{U}_2$ is smooth and everywhere of dimension $n$.
 Moreover, at all points
of $E_2\cap \hat{U}_2$
$t$ is a local equation for the divisor $E_2=(\theta\circ \nu)^{-1}(E_1)$. 
\end{lemma}

\begin{proof}
Let $(x,s,t)$ be a point of $\hat{U}_2$ with $x\in W$. 
We consider $W$ as a subset of an affine space $\mathbb{A}^M$. Let
$x_1\ldots x_m$ be the local coordinates on $\mathbb{A}^M$, and let
$h_1,\ldots h_k$ be local equations for $W$ at $x$. The assumption that $x$
should be a smooth point of $W$ at which $Q_i$ and $Q_j$ are smooth and meet
transversally means that the vectors 

\[
\left\{
\left(
\begin{array}{c}
 \frac{\partial h_1}{\partial x_1} \\ \vdots \\ \frac{\partial h_1}{\partial x_m}
\end{array}
\right), 
\dots,
\left(
\begin{array}{c}
 \frac{\partial h_k}{\partial x_1} \\ \vdots \\ \frac{\partial h_k}{\partial x_m}
\end{array}
\right),
\left(
\begin{array}{c}
 \frac{\partial f}{\partial x_1} \\ \vdots \\ \frac{\partial f}{\partial x_m}
\end{array}
\right),\left(
\begin{array}{c}
 \frac{\partial g}{\partial x_1} \\ \vdots \\ \frac{\partial g}{\partial x_m}
\end{array}
\right)
\right\}
\]

are linearly independent. (The implicit evaluations at $x$ have been omitted
for legibility's sake.)

$\hat{U}_2$  is a subset of $\mathbb{A}^M\times (\mathbb{A}^1\setminus \{0\})\times
\mathbb{A}^1$ given by the set of equations 
\[h_1\ldots h_k, t^a-s^dg, t^b-s^cf.\] 
It follows from the Jacobian criterion that $\hat{U}_2$ is
smooth and of dimension $n$ everywhere. 
Moreover, \[t(y)=0\Leftrightarrow f(y)=g(y)=0\]
and hence $E_2$ is set-theoretically given by the equation 
$t=0$. The Jacobian criterion also shows that $dt\neq 0$ in 
$\Omega_{\hat{U}^2}^1$ at any point 
$x\in\hat{U}_2$
 and it follows that $t$ is a local equation for $E_2$. 
\end{proof}

Let us show that  
 $\hat{U}_2$ is integral. It is enough to show that it is not a disjoint 
union of disconnected components. But this follows from the fact that $\nu$ 
is one-to-one and that every component of $\hat{U}_2$ has dimension 
$n={\rm dim}(\hat{U}_1)$. 

We now show that the normalisation map factors through $\nu:\hat{U}_2
\rightarrow\hat{U}_1$.
Over the points where 
$g\neq 0$ we can write \[t=\frac{t^{bd}}{t^{ac}}=
\frac{s^{cd}f^d}{s^{dc}{g^c}} \in A(\hat{U}_1),\] so there
is an open set over which $\nu$ is an isomorphism. Hence there is an inclusion
$A(\hat{U}_2)\subset K(\hat{U}_1)$. 
Moreover, $\hat{U}_2$ is
integral over $\hat{U_1}$. It follows that 
there is a factorisation 
\[\hat{U}'\stackrel{\phi}{\rightarrow} \hat{U}_2
\stackrel{\nu}{\rightarrow} \hat{U}_1\]
such that $\nu\circ\phi=\mu$, and such that $\hat{U}'$ is the normalisation 
of $\hat{U}_2$. But since $\hat{U}_2$ is smooth and hence normal, $\phi$ is an 
isomorphism.  

It follows that 
$\hat{U}'$ has a unique exceptional divisor over $W$, $\phi^{-1}(E_2)$, and 
hence that $\hat{X}'$ indeed contains a unique exceptional divisor, 
$\overline{\phi^{-1}(E_2)}$, which we denote 
by $E$.
\end{proof}

We will now show that the $(a,b)$ blow-up has good properties. 
\begin{lemma}\label{lem3} 
Suppose that $Q_i$ is of type 1 and $Q_j$ is of type 2.
For a suitable choice of (a,b), the coefficient of $E$ in 
$\pi'^* D_1$ is the same as its coefficient in $\pi'^* D_2$. 
\end{lemma} 

\begin{proof}
In the above notation, 
$t$ is a local equation for $E$ at a generic point of
$E$. Let $f$ be a local equation for $Q_i$. We have seen above that
at a generic point of $E$, $t^b=s^cf$, so $E$ appears with coefficient $b$ in
$\pi'^*(Q_i)$. Likewise, $E$ appears with coefficient $a$ in 
$\pi'^*(Q_j)$. 

Now, since $Q_i$ is of type 1 and $Q_j$ is of type 2 we can write 
\[
D_1 \equ  M+ c_1 Q_i+ F_1, D_2= M+c_2 Q_j+ F_2
\]
where 
\begin{itemize}
\item $M$ is the minimum of $D_1$ and $D_2$, 
\item the $c_i$'s are positive rationals,
\item $F_1$ and $F_2$ are divisors whose support does not contain 
$Q_i\cap Q_j$.
\end{itemize}
In particular, $F_1$ and $F_2$
do not contribute to the coefficient of $E$ in
$\pi'^*(D_i)$.
It is 
therefore enough to require $c_1b=c_2a$. In other words, by picking
$(a,b)$ to be  the unique pair of coprime positive 
integers such that $a/b=c_1/c_2$, we can arrange the required coefficients to be equal.
\end{proof}

We now need the following proposition.
\begin{proposition}\label{pr4}
The divisors $\overline{Q}_i$ and $\overline{Q}_j$ do not meet in $\hat{X}'$.
\end{proposition}

\begin{proof}
It will be enough to show that
$\overline{Q}_i$ and $\overline{Q}_j$ do not meet in $\hat{X}$. But for any $k$
$\overline{Q}_i\cap \hat{U}_k$ is contained in the set given by 
$U=0$ and $\overline{Q}_j\cap\hat{U}_k$ is contained in the set given by 
$V=0$, which are disjoint. 
\end{proof}

Henceforth, we will call any
$(a,b)$ blow-up along $(Q_i, Q_j)$ such that the
coefficient of $E$ is the same in $\pi^*(D_1)$ as in $\pi^*(D_2)$ 
a separating
blow-up for $(i,j)$. In particular, if $(\hat{X}', \pi')$ is a separating 
blow-up for $(Q_i, Q_j)$ then $\pi'$ has a unique 
exceptional divisor of type 0 and that $\overline{Q}_i$ and $\overline{Q}_j$ 
do not meet in $\hat{X}'$. 

\begin{proof}[Proof of Proposition~\ref{prop:pr2}.]
Choose $(a,b)$ such that the conditions of Lemma~\ref{lem3} are fulfilled,  
let $\tilde{X}=\hat{X}'$ and $\mu=\pi'$ for this pair $(a,b)$. According to Proposition~\ref{pr4}, the 
morphism $\mu$ is a separating blow-up for the pair $(Q_i,Q_j)$ provided the assumptions \ref{assn:1-4} are satisfied. 

We start with proving that $\hat{X}'$ is $\Q$-factorial. We have that 
\[
\Q{\rm Weil}(\hat{X}')=\pi^*(\Q{\Weil}(X))  \oplus \spn{E} \ .
\] 
We are done if we can show that $E$ is a $\Q$-Cartier divisor. 
It will be enough to produce a Cartier divisor $L$ on $\hat{X}$ 
such that (set-theoretically) $\Supp (L)= E_1$. Indeed, the pull-back
$n^*(L)$ is then  a Cartier divisor on $\hat{X}'$ whose support is contained in
$n^{-1}(E_1)$. But this set  contains only one prime divisor,
$E$, so the Weil divisor associated to $\pi'^*(E_1)$  
is necessarily a multiple of $E$. 

 We now construct $L$ as follows. Consider the covering of $\hat{X}$ by 
the sets $\hat{U}_k^1$ and $\hat{U}_k^2$ given by
\[
\hat{U}_k^1 \equ \{ (x, [U:1])\in \hat{U}_k\}\] \[\hat{U}_k^2=\{ ( x, [1:V])\in
\hat{U}_k\}\ .
\] 
We choose the Cartier divisor given by
 $g_k^b$ on $\hat{U}_k^1$ and $f_k^a$ on 
$\hat{U}_k^2$. It is immediate that the support of this Cartier divisor is
$E_1$, hence $\tilde{X}=\hat{X}'$ is $\Q$-factorial.

Condition (1) of \ref{assn:1-4}  is therefore inherited. Conditions (2)-(4) will quickly follow
from the following lemma.
\begin{lemma}\label{lem4}
Consider divisors $\overline{Q}_{k_1}\ldots, \overline{Q}_{k_m}$ in
$\hat{X}'$. The intersection \[\overline{Q}_{k_1}\cap\ldots\cap
\overline{Q}_{k_m}\cap E\] is of codimension $\geq m+1$.
\end{lemma}

\begin{proof}
Since $\overline{Q}_i\cap \overline{Q}_j=\emptyset$ we can assume that either
\begin{enumerate}
\item $i,j \neq \{ k_1,\ldots, k_m\}$,
\item $i\in \{ k_1,\ldots, k_m\}$, $j\not\in \{ k_1\ldots, k_m\}$
\end{enumerate}
We consider first the

\noindent
{\bf  Case 1.}
We have that 
\[\overline{Q}_{k_1}\cap\ldots \cap \overline{Q}_{k_m}\cap E\subset
\pi'^{-1}(Q_{k_1}\cap...\cap Q_{k_m}\cap Q_i\cap Q_j).\] But
$Q_{k_1}\cap\ldots \cap Q_j$  
has codimension $(m+2)$, so 
$\pi'^{-1}(Q_{k_1}\cap...\cap Q_{k_m}\cap Q_i\cap Q_j)$ has codimension $\geq m+1$. 

\noindent
{\bf Case 2.}
We assume without loss of generality that $i=k_1$. We then have that 
\[(\overline{Q}_{k_1}\cap\ldots \cap \overline{Q}_{k_m}\cap E )\subset
(\overline{Q_i}\cap E)\cap 
\pi'^{-1}(Q_{k_2}\cap...\cap Q_{k_m})\] and
\[ (\overline{Q_i}\cap E)\cap 
\pi'^{-1}(Q_{k_2}\cap...\cap Q_{k_m})\subset
(E\cap \overline{Q}_i)\cap
\pi'^{-1}(Q_{k_2}\cap...\cap Q_{k_m}\cap Q_i\cap Q_j).\]
But the map 
\[ 
\pi': E\cap \overline{Q}_i\rightarrow Q_i\cap Q_j
\] 
is finite-to-one,
so the 
codimension of 
\[
(E\cap \overline{Q}_i)\cap \pi'^{-1}(Q_{k_2}\cap...\cap Q_{k_m}\cap Q_i\cap Q_j)
\] 
is at least $m+1$ and 
\[
{\rm codim}(Q_{k_2}\cap...\cap Q_{k_m}\cap Q_i\cap Q_j)\geq m+1\ .
\] 
\end{proof}

But now, every irreducible component of
$\overline{Q}_{k_1}\cap \ldots \cap
\overline{Q}_{k_m}$ is of codimension at most $m$, since it is
an intersection of $m$ divisors in a $\Q$-factorial 
normal variety. It follows that 
\[(\overline{Q}_{k_1}\cap \ldots\cap
\overline{Q}_{k_m})\cap(\pi^{-1}(X\setminus (Q_i\cap Q_j)))\cong 
Q_{k_1}\cap\ldots\cap Q_{k_m}\cap(X\setminus Q_i\cap Q_j)\] 
is a dense open subset of 
$\overline{Q}_{k_1}\cap 
\ldots \cap \overline{Q}_{k_m}$. Hence (2), (3) and (4)
hold for $\hat{X}'$. 
This completes the proof of Proposition \ref{prop:pr2}.
\end{proof}

\end{document}